\documentclass[11pt]{amsart}
\usepackage{graphicx}
\usepackage[all]{xy}
\usepackage{amscd}
\usepackage{amssymb}
\usepackage{amsmath}
\usepackage{amsthm}
\usepackage{amsfonts}
\usepackage{graphics}
\usepackage{epsfig,psfrag}
\usepackage[english]{babel}
\usepackage{latexsym}
\usepackage{mathrsfs}

\newtheorem{theorem}{Theorem}[section]
\newtheorem{corollary}[theorem]{Corollary}
\newtheorem{lemma}[theorem]{Lemma}

\newtheorem{proposition}[theorem]{Proposition}
\newtheorem{question}[theorem]{Question}

\theoremstyle{definition}

\newcommand{\RR}{{\mathbb R}}

\newcommand{\ZZ}{{\mathbb Z}}

\newcommand{\id}{{\rm id}}

\newcommand{\C}{{\mathcal C}}

\theoremstyle{remark}

\begin{document}

\title[Equivariant embeddings  of manifolds  into Euclidean spaces]
{Equivariant embeddings  of manifolds  into Euclidean spaces}
 
\author{Zhongzi Wang}
\address{Department of Mathematical Sciences, Tsinghua University, Beijing, 100084, CHINA}
\email{wangzz18@mails.tsinghua.edu.cn}

\subjclass[2010]{Primary 57M60; Secondary 57R40, 20C30, 20H10}

\keywords{Finite group actions on manifolds, Equivariant embeddings in Euclidean spaces, Fixed point sets, Hurwitz homomorphism}


\begin{abstract} Suppose a finite group $G$ acts on a manifold $M$. 
By a theorem of Mostow, also Palais, there is a $G$-equivariant embedding of 
 $M$ into the $m$-dimensional Euclidean space $\RR^{m}$ for some $m$. We are interested in some explicit bounds of such $m$.
  
 First we provide an upper bound: there exists a $G$-equivariant embedding of 
 $M$ into $\RR^{d|G|+1}$, where $|G|$ is the order of $G$ and $M$ embeds into $\RR^d$. 
 Next we provide a lower bound for finite cyclic group action $G$: If there are $l$ points having pairwise co-prime lengths of $G$-orbits  greater than $1$ and there is a  
 $G$-equivariant embedding of 
 $M$ into $\RR^{m}$, then $m\ge 2l$.  

Some applications to surfaces are given.
\end{abstract}

\date{}
\maketitle

\section{Introduction}
In this note, we assume that $M$ is a compact and connected polyhedron  and $G$ is a finite group which acts faithfully on $M$. We often call such $M$ and $G$ a pair $(M,G)$.
Recall  $\RR^m$ is the $m$-dimensional Euclidean space,  and $SO(m)$ is the $m$-dimensional special orthogonal group, which  
 acts on $\RR^m$ canonically.
We use  $|G|$ to denote the order of a finite group $G$ and $F_g$ 
to denote the closed orientable surface of genus $g$.

{\bf Definitions:} 
Let $(M, G)$ be a pair with the action of $G$ on $M$ given by a representation 
$\rho : G \to \text{Homeo}(M)$, where $\text{Homeo}(M)$ is the group of homeomorphisms on $M$. 
Call an embedding $e: M \to \RR^m$ $G$-equivariant, if there is an orthogonal representation $\tilde \rho: G\to SO(m)$ such that 
$$e\circ \rho (g)=\tilde \rho(g) \circ e$$
for any $g\in G$. 

\vskip 0.2truecm



When $G$ is a finite cyclic group generated  by a periodic map $f$ on $M$, we often  call
$G$-action as $f$-action and $G$-equivariant as $f$-equivariant.


The existence of $G$-equivariant embedding for pair $(M, G)$ follows from the work of Mostow, also Palais,  see \cite[Theorem 6.1]{Mos} and \cite[Theorem]{Pa}.

\begin{question} For a given pair $(M, G)$, to find some concrete small $n$, or stronger,  the minimal $n$,  so that there is a $G$-equivariant embedding $M\to \RR^n$.
\end{question}


There are some systematic
studies on  $G$-equivariant embeddings for graphs and surfaces into $\RR^3$ and $S^3$, see \cite{Cos}, \cite{FNPT}, \cite{WWZZ1}, \cite{WWZZ2} and the references therein. Those studies rely on the geometry and topology of 3-manifolds developed in the last several decades, as well as on our 3-dimensional intuition.  Once we know that there is a $G$-equivariant embedding  $M\to \RR^3$, then the integer $3$ is often the minimal
$n$ for the pair  $(M, G)$ in Question 1.1, since usually graphs and surfaces themselves can not be embedded into $\RR^2$.
  
If there is no  $G$-equivariant embedding $M\to \RR^3$ for a  pair $(M,G)$,
then Question 1.1 becomes  more complicated. 
 Some  results for $G$-equivariant embeddings $M\to \RR^n$ to high dimensional Euclidean space appear recently, see \cite{Zi}, \cite {Wa}.

In this note, we try to give some general explicit bounds for Question 1.1.

First we give an upper bound in terms of $|G|$, and the dimension of Euclidean space that $M$ embeds.  We state the smooth version.
The  topological version follows by just ignoring the smoothness in both the statement and its proof.

\begin{proposition}\label{upperbound} Suppose $M$ is a closed smooth
manifold and there is a smooth finite group action $G$ on $M$. If there is a smooth embedding of 
$M$ into $\RR^d$, then there exists a $G$-equivariant smooth embedding of $M$ into $\RR^{d|G|+1}$.
\end{proposition}

Then we give a lower bound for finite cyclic group actions in terms of periods of periodic points for periodic maps.
Suppose $G=\ZZ_n$ is generated by  a periodic map $f$ of order $n$ on $M$. Under $f$-action each point of $M$ has its $f$-orbit with length dividing $n$.

\begin{proposition}\label{lowerbound}
Suppose $f$ is a periodic map on $M$ and  there are $l$ points having pairwise 
 co-prime lengths of $f$-orbits greater than $1$.
If there is an $f$-equivariant embedding $e: M\to \RR^m$, then $m\ge 2l$.
\end{proposition}

Finite group actions on surfaces are keeping to be an active topic since
the work of Hurwitz \cite{Hu}.
Some applications to finite group actions on surfaces are given below.

By Hurwitz theorem the order of any orientation-preserving finite group action on $F_g$, $g>1$, is bounded by $84(g-1)$ \cite{Hu}.  Since $F_g$ embeds into $\RR^3$ and $  84(g-1)\times 3= 252(g-1)$,  by Proposition \ref{upperbound} we have

\begin{corollary}\label{252}
 For any orientation-preserving finite group action $G$ on $F_g$, $g>1$, 
 $F_g$ can be $G$-equivariantly embedded into $\RR^{252(g-1)+1}$. 
\end{corollary}

The next result is a corollary of Proposition \ref{lowerbound}.

\begin{corollary}\label{2l} For any given integer $m>0$, there is a periodic map $f$ on a closed orientable surface $F$ 
such that there is no $f$-equivariant embedding $e: F\to \RR^m$. 
\end{corollary}

Proposition \ref{upperbound}, Proposition \ref{lowerbound} and Corollary \ref{2l} will be proved in Sections 2, 3, 4 respectively.

\vskip 0.3 true cm

{\bf Acknowledgement:} 
We thank the referee for the advice.


\section{An upper bound for finite group actions}


\begin{proof}[Proof of Proposition \ref{upperbound}]  Suppose the smooth action of $G$ on $M$ is given
by the representation 
$$\rho : G \to \text{Diff} (M)$$
where $\text{Diff}(M)$ is the group of diffeomorphisms on $M$. 
and the smooth embedding of $M$ into $\RR^d$ is given by 
$$e:M \to \RR^d.$$



We may assume $G=(\{1,2,..., n\},*)$.
With these notations we define a map $$\tilde e : M\to \RR^{d|G|}$$ as follows: for any $x \in M$, 
$$\tilde e (x)=(e(\rho(1)(x)),e(\rho(2)(x)),...,e(\rho(n)(x))).$$

We will prove $\tilde e$ is a smooth  embedding.  Since $\rho(1)$ is a diffeomorphism of $M$ onto itself, the first component $\tilde e_1$ of the mapping $\tilde e$, where $\tilde e_1=e\circ\rho(1)$,  is a smooth embedding. By the same reason, all components $\tilde e_j$ of $\tilde e$, $\tilde e_j=e\circ\rho(j)$ are smooth. Hence $\tilde e$ is smooth and injective. The rank of the Jacobi matrix of $\tilde e$ at each point $x \in M$ is not less than the rank of the Jacobi matrix of $\tilde e_1$, the first component of $\tilde e$, which equals the dimension of $M$ since $\tilde e_1$ is an embedding. Thus $\tilde e$ is also an immersion, and is hence a smooth embedding. 

We have smoothly embedded $M$ into $\RR^{d|G|}$ and we next prove that $\tilde e$ is a $G$-equivariant embedding. Recall that $O(m)$ is the $m$-dimensional orthogonal group. We construct a natural group action $\tilde  \rho$ of $G$ on $\RR^{d|G|}$, that is to define an embedding 
$$\tilde \rho: G\to O(d|G|)$$ by 
$$\tilde \rho(j)(y_1,y_2,...,y_n)=(y_{1*j},y_{2*j},...,y_{n*j}),$$ where each element $\tilde \rho(j)$ acts as an orthogonal transformation. It suffices to show that the image of $M$ under the embedding $\tilde e$ 
is invariant under the group action $\tilde  \rho$ on $\RR^{d|G|}$ and the actions $\tilde \rho$ and $\rho$ are commutative by the embedding $\tilde e$. If $y= \tilde e(x)$ where $x \in M$, then 

\begin{eqnarray*}
\tilde\rho(j)(\tilde e(x))&=&
\tilde\rho(j)(e(\rho(1)(x)),e(\rho(2)(x)),...,e(\rho(n)(x)))\\
&=&(e(\rho(1*j)(x)),e(\rho(2*j)(x)),...,e(\rho(n*j)(x)))\\
&=&(e(\rho(1)(\rho(j)(x))),e(\rho(2)(\rho(j)(x))),...,e(\rho(n)(\rho(j)(x))))\\
&=&\tilde e(\rho(j)(x)),
 \end{eqnarray*}
which implies $\tilde \rho(j)(y)\in \tilde e(M)$ and thus $\tilde e(M)$ is invariant under each $\rho(j)$, and is hence invariant under $\tilde \rho$ of $G$. 
We conclude $$\tilde \rho(g)\circ \tilde e=\tilde e\circ\rho(g)$$ for all $g\in G$.
 Moreover the restriction of the action $\tilde \rho$ on $\tilde e(M)$ is the
action $\rho$.  

If $\tilde \rho(G)\subset SO(d|G|)$, then we finish the proof. 
Otherwise let 
$$\tilde \rho^*: G\to SO(d|G|+1)$$ 
be a group homomorphism defined as 
$$\tilde \rho^*(g)=(\tilde \rho(g), \text{det}(\tilde \rho(g))\text{Id}_\RR),$$
where $\text{det}(\tilde \rho(g))$ is 1 if $\tilde \rho(g)$ is orientation preserving and $-1$ otherwise.

Now let $$\tilde e^*: M\to \RR^{d|G|+1} =\RR^{d|G|}\times \RR$$
be defined as 
$$\tilde e^*(x)=(\tilde e(x), 0).$$
Then $\tilde e^*$ is an embedding. Moreover 
 
$$\tilde \rho^*(g)\circ \tilde e^*(x)=\tilde \rho^*(g) (\tilde e(x),0)=(\tilde \rho(g)\circ \tilde e(x),0)=(\tilde e \circ  \rho(g)(x),0)= \tilde e^*\circ\rho(g)(x)$$ for all $x\in M$ and $g\in G$.
Proposition \ref{upperbound} is proved.  \end{proof}

\section{A lower bound for finite cyclic group actions}

\begin{lemma}
{\it Let  $A$  be an orientation-preserving isometry of  $\RR^m$  with a fixed point (i.e., $A\in SO(m)$).
Assume that there are $s$  points in  $\RR^m$  having pairwise coprime lengths of $A$-orbits greater than $1$.
Then  $m\ge 2s$.}
\end{lemma}

\begin{proof}
Denote by $w_1,\ldots,w_s>1$ the pairwise coprime lengths.
Take any $j=1,\ldots,s$.
Take $x_j\in\RR^m-0$ such that $A^{w_j}x_j = x_j$.
Define
\[
\overline x_j = \frac{x_j + A x_j + \ldots + A^{w_j-1}x_j}{w_j}.
\]
Then $A\overline x_j = \overline x_j$.
So for $u_j := x_j - \overline x_j$ we have
$u_j + A u_j + \ldots + A^{w_j-1}u_j = 0$.
Hence
\[
\det \left(\id + A + \ldots + A^{w_j-1} \right) =
\det\prod_{k=1}^{w_j-1}\left( A - e^{2\pi i k/w_j} \id \right) = 0.
\]
So $A$ has an eigenvalue $e^{2\pi ik_j/w_j}$ for some $1\le k_j\le w_j-1$.
Since $w_1,\ldots,w_s$ are pairwise coprime, all these $s$ eigenvalues are pairwise distinct.

Since $A$ is a real operator, for odd $w_j$ the conjugate
$e^{-2\pi ik_j/w_j}$ is also an eigenvalue of $A$.
Observe that $n$ is not smaller than the number of eigenvalues of $A$.
 
If every $w_j$ is odd, then $n\ge 2s$.

If some $w_j$ is even, then such $j$ is unique.
If further $n<2s$, then $n=2s-1$, and the eigenvalues of $A$ are $e^{\pm 2\pi ik_t/w_t}$ for $t\ne j$, together with the eigenvalue $-1$ corresponding to the even $w_j$.
This is impossible because the product of the eigenvalues is $\det A > 0$.
\end{proof}

\begin{proof}[Proof of Proposition \ref{lowerbound}]
Suppose $f$ is a periodic map on $M$ and there are $l$ points having pairwise coprime lengths of $f$-orbits greater than $1$. Suppose $A: \RR^m\to \RR^m$, $A\in SO(m)$,  
extends $f$ for some embedding $e: M\to \RR^m$.
Then there are $s$ points having pairwise coprime lengths of $A$-orbits greater than $1$.
 By Lemma 3.1, $m\ge 2l$.
\end{proof}

\section{Applications to surfaces}

\begin{proof}[Proof of Corollary \ref{2l}] It follows from Proposition \ref{lowerbound} and the following Proposition \ref{example}. \end{proof}

\begin{proposition}\label{example} For each given integer $l>0$, there is a  periodic map $f$  on a closed orientable surface $F$ so that there are $l$ points having pairwise coprime lengths of $f$-orbits greater than $1$.
\end{proposition}

\begin{proof}
We will use the Hurwitz type construction to get such a periodic map $f$.
The theories of 2-orbifolds, especially those of fundamental groups and covering spaces,  are parallel to those of 2-manifolds, see \cite{Sc}.

Let $p_1, ..., p_l$ be the first $l$ prime numbers
and $P_l=p_1p_2...p_l$ be their product, and $\delta_j=\frac {P_l}{p_j}.$
Let $O_l$ be the 2-orbifold having underlying space  the 2-sphere, and two singular points of index $\delta_j$ for each $j\in 
\{1,2,....,l\}$. Then we have its orbifold fundamental group presentation

$$\pi_1(O_l)=\left< x_1, {x'}_1, x_2, {x'}_2,..., x_l, {x'}_l\vert \prod_{j=1}^l x_j{x'}_j=1,\,\, x_j^{\delta_j}={x'}_j^{\delta_j}=1,\,\, j\in \{1,2,....l\} \right> .$$

Now define a homomorphism 
$$\phi : \pi_1(O_l)\to \ZZ_{P_l}=\left<f | f^{P_l}=1\right>$$
by
$$\phi(x_j)=f^{p_l}, \,\,\,\phi(x'_j)=f^{-p_l},$$
where $f$ is a generator of the cyclic group $\mathbb{Z}_{P_l}$ for this moment.

Since $p_1, p_2, ..., p_l$ are pairwise co-prime, $\phi$ is surjective by Chinese Remainder Theorem. 

Since all torsion subgroups  in $\pi_1(O_l)$ are conjugated to those finite cyclic groups $\left< x_j \right>$ and $\left< x'_j \right>$, 
$j=1,..., l$ \cite{Gr}, and those  $\left< x_j \right>$ and $\left< x'_j\right>$  inject into  $\ZZ_{P_l}$ under $\phi$, we conclude that
the kernel of $\phi$ is torsion free \cite{Ha}, and hence  the kernel of $\phi$ is the fundamental group of a surface  $F$.  Hence we have a short exact sequence
$$1\to \pi_1(F)\to \pi_1(O_l)\to \left<f | f^{P_l}=1\right>\to 1,$$
where $f$ acts on $F$ as a cyclic group action of order 
$P_l$, and we have a cyclic branched cover 
$$p: F\to F/f= O_l.$$ 
Since the orbifold $O_l$ is closed and orientable, the surface $F$ is  closed and orientable.
There is a one-to-one correspondence between singular points of index $\delta_j$ on $O_l$ and the $f$-orbits on $F$
of length $p_j$.  We conclude that there are $l$ points having the first $l$ primes as their lengths of $f$-orbits. Noting that the first $l$ primes are pairwise coprime and greater than $1$, we have proved Proposition \ref{example}.
\end{proof}

\end{document}